\documentclass[reqno,12pt]{amsart}

\usepackage{epsf}
\usepackage{graphics}
\usepackage{amssymb}
\usepackage{amsmath}

\date{}

\theoremstyle{plain}
\newtheorem{theorem}{Theorem}

\newtheorem{lemma}{Lemma}
\newtheorem{proposition}{Proposition}

\theoremstyle{definition}

\theoremstyle{remark}
\newtheorem*{example}{Example}

\newtheorem*{remark}{Remark}

\def\C{{\mathbb C}}
\def\N{{\mathbb N}}
\def\Z{{\mathbb Z}}
\def\Q{{\mathbb Q}}
\def\R{{\mathbb R}}

\title{Scissor Equivalence for Torus links} 

\author{Sebastian Baader}

\begin{document}

\begin{abstract} This article is about a natural distance function induced by smooth cobordisms between links. We show that the cobordism distance of torus links is determined by the profiles of their signature functions, up to a constant factor.
\end{abstract}

\maketitle

\section{Introduction}

The Thom conjecture asserts that algebraic curves in $\C^2$ are genus-minimising: the intersection of a smooth algebraic curve defined by a polynomial $f \in \C[x,y]$ with the closed unit ball in $\C^2$ has minimal genus among all smoothly embedded surfaces with the same boundary link on the unit sphere, provided the intersection of the curve with the unit sphere is transverse. More generally, the transverse intersection of a smooth algebraic curve with the compact domain enclosed by two spheres of different radii in $\C^2$ is a smooth cobordism of minimal genus between the two boundary links. In general, two algebraic links are not connected by an algebraic cobordism, as we will shortly see. However, we may still look for a minimal smooth cobordism between pairs of algebraic links. Let us define the cobordism distance $d_{\chi}(K,L)$ between two oriented links $K$, $L \subset S^3$ as the maximal Euler characteristic among all smooth cobordisms without closed components in $S^3 \times [0,1]$ connecting the links $K$ and $L$. The main goal of this paper is to determine the cobordism distance between torus links, up to a constant factor.

\begin{theorem} There exists a constant $\gamma \geq 1$, such that the following inequalities hold for almost all pairs of torus links $T(a,b)$, $T(c,d)$:
$$\tau(T(a,b), T(c,d)) \leq d_{\chi}(T(a,b), T(c,d)) \leq \gamma \, \tau(T(a,b), T(c,d)).$$
\end{theorem}

Here the quantity $\tau$ is defined for all pairs of oriented knots $K$ and $L$, as follows:
$$\tau(K,L)=\max\{|\chi(K)-\chi(L)|,|\sigma_{\omega}(K)-\sigma_{\omega}(L)||\omega \in S^1 \},$$
where $\chi$ and $\sigma_\omega$ denote the maximal Euler characteristic and the signature functions, respectively. These will be discussed in Section~3.
Throughout this paper, we will often have to exclude a few exceptional torus links, finite in number for each braid index. It is in this sense that the expression `for almost all torus links' is to be understood. It is actually possible to write out the exceptional torus links and the constant $\gamma$ of Theorem~1, but this would be at the expense of legibility. We do not know if there exists a constant such that the statement is true for all torus links.  

The proof of Theorem~1 is based on a construction of effective cobordisms. For this purpose, we will introduce three special types of smooth cobordisms called scissor cobordisms between torus links. Two of them, cutting and gluing, enable us to give a coarse estimate for the cobordism distance, which may be thought of as a relative version of the Thom conjecture. Let us define a function $f: \N^4 \to \N$ by the equation
$$d_{\chi}(T(a,b), T(c,d))=|\chi(T(a,b))-\chi(T(c,d))|+f(a,b,c,d).$$
It is easy to see that the function $f$ is bounded above by a quadratic expression. There is even a linear bound.

\begin{theorem} The following inequalities hold for all $a,b,c,d \in \N$:
$$0 \leq f(a,b,c,d) \leq 2(a+b+c+d).$$
\end{theorem}

Surprisingly, the lower bound is sharp for all pairs of torus links of type $T(ab,c)$, $T(a,bc)$, i.e. $f(ab,c,a,bc)=0$ (Proposition~1, Section~2). This fact involves the third scissor cobordism just mentioned, and is the key ingredient in the proof of Theorem~1. 

Sections~2 and~3 are devoted to the proofs of Theorems~1 and~2, in the reverse order. In Section~4 we exhibit two families of pairs of torus links showing that there is essentially no better linear bound than the one of Theorem~2 for the function $f$. We conclude with an interesting application of Theorem~2 concerning the stable $4$-genus of knots, defined by Livingston in~\cite{Li}. The stable $4$-genus defines a semi-norm on the knot concordance group. In Section~5 we will see that the restriction of this semi-norm to the span of pairs of torus knots has extremely flat unit balls.

\section{Scissor equivalence}

Torus links are prototypes of links of isolated singularities. They can be described as intersections of plane algebraic curves of the form $\{(z,w) \in \C^2| z^p-w^q=0\}$ with the unit sphere in $\C^2$. A slight perturbation of these curves will not affect the corresponding link types, denoted by $T(p,q)$. Thus torus links bound pieces of smooth algebraic curves in the $4$-ball. The Euler characteristic of these pieces of curves equals $-pq+p+q$ and is known to be maximal among all smooth oriented surfaces without closed components in the $4$-ball bounding the same link, due to the Thom conjecture \cite{KM}, \cite{Ra}. As we mentioned in the introduction, two algebraic links are not necessarily connected by an algebraic cobordism in $S^3 \times [0,1]$. For example, the two torus knots $T(2,13)$ and $T(4,5)$ have minimal genus $6$, thus an algebraic cobordism connecting these would have genus zero. This is impossible since their signature invariants are not equal.

As the title indicates, we will use scissor equivalence techniques for rectangles in order to construct cobordisms between torus links. All these cobordims will be compositions of simple $1$-handles of two types, as shown in Figure~1. We call the corresponding moves smoothing resp. saddle move.

\begin{figure}[ht]
\scalebox{0.8}{\raisebox{-0pt}{$\vcenter{\hbox{\epsffile{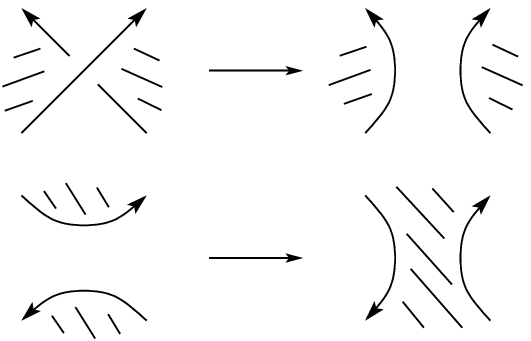}}}$}} 
\caption{}
\end{figure}

Every $1$-handle contributes $-1$ to the Euler characteristic of a cobordism. As an application, we may split a torus link of type $(a+b,c)$ into the disjoint union of two torus links of type $(a,c)$ and $(b,c)$  by a cobordism of Euler characteristic $-c$. Figure~2 illustrates this fact for the case $a=3$, $b=4$, $c=6$, where $6$ crossings have to be smoothed between the third and the fourth strand. All braids are to be closed in the standard way. 

\begin{figure}[ht]
\scalebox{0.8}{\raisebox{-0pt}{$\vcenter{\hbox{\epsffile{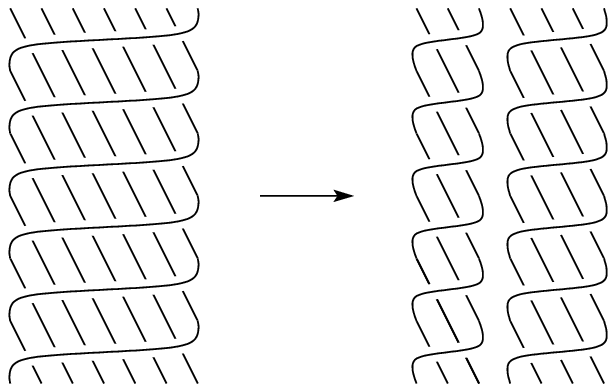}}}$}} 
\caption{}
\end{figure}

\begin{proof}[Proof of Theorem 2]
The cobordism distance of two links $K,L \subset S^3$ without slice components is bounded below by the difference of their Euler characteristics: \footnote{Brendan Owens pointed out that the triangle inequality does not hold for the cobordism distance in the presence of slice components, since in that case the composition of two cobordisms may contain spheres. This cannot happen for torus links, as their components are positively linked.}
$$|\chi(K)-\chi(L)| \leq d_{\chi}(K,L).$$
Combining this with the Thom conjecture, 
$$\chi(T(p,q))=-pq+p+q,$$
we obtain the first inequality of Theorem~2: 
$$|-ab+a+b+cd-c-d| \leq d_{\chi}(T(a,b), T(c,d)).$$

The second inequality will be proved by induction on $a+b+c+d \in \N$. The starting case $a=b=c=d=1$ is trivial.
Without loss of generality, we may suppose $a \leq d$. According to the relative sizes of $b$ and $c$, we will distinguish two cases:

\medskip
1. $b<c$ \quad \scalebox{1.0}{\raisebox{-22pt}{$\vcenter{\hbox{\epsffile{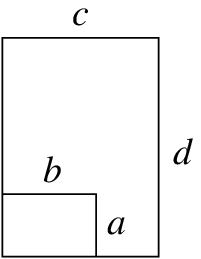}}}$}} 

\medskip
In this case the torus link $T(a,b)$ can be obtained from the standard diagram of the torus link $T(c,d)$ by smoothing an appropriate part of the crossings, as shown in Figure~3 for $a=3$, $b=4$, $c=7$, $d \geq 3$ arbitrary. The number of crossings to be smoothed equals
$$(c-1)(d-a)+(a-1)(c-b)=-ab+cd+a+b-c-d,$$
which is precisely the lower bound.

\begin{figure}[ht]
\scalebox{0.8}{\raisebox{-0pt}{$\vcenter{\hbox{\epsffile{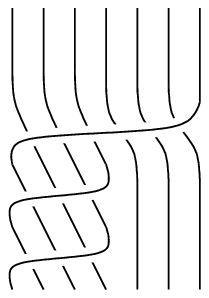}}}$}} 
\caption{}
\end{figure}

\medskip
2. $b \geq c$ \quad \scalebox{1.0}{\raisebox{-26pt}{$\vcenter{\hbox{\epsffile{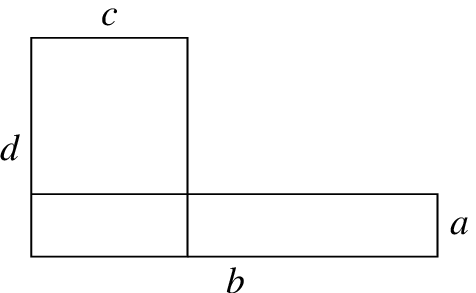}}}$}} 

\medskip
Here we may split off a torus link of type $(a,c)$ from the links $T(a,b)$ and $T(c,d)$ by a cobordism of Euler characteristic $-a$ resp. $-c$. The remaining links are torus links of type $(a,b-c)$ resp. $(c,d-a)$. Using the inductive hypothesis, the cobordism distance of the links $T(a,b-c)$ and $T(c,d-a)$ is bounded above by
$$A'+2(a+b-c+c+d-a)=A'+2b+2d,$$
where
\begin{eqnarray*}
A' &=& |(a-1)(b-c-1)-(c-1)(d-a-1)| \\
   &=& |(a-1)(b-1)-(c-1)(d-1)+c-a| \\
   &\leq& A+a+c. \\
\end{eqnarray*}

Altogether, there is a cobordism of Euler characteristic of absolute value at most
$$A'+2b+2d+a+c \leq A+2(a+b+c+d)$$
between the links $T(a,b)$ and $T(c,d)$.
\end{proof}

In view of the Milnor conjecture, it is natural to ask whether the estimates of Theorem~2 hold for the Gordian distance of torus knots, i.e. the minimal number of crossing changes needed to pass from one torus knot to another~\cite{HU}. This may be more difficult since the cobordism distance of two knots provides a lower bound~\cite{Ka}, but not an upper bound, for their Gordian distance. We do not even know if the correction term (corresponding to $f$ in Theorem~2) is a sub-quadratic function for the Gordian distance of torus knots.

The following proposition is a sharpened version of Theorem~2 for special torus links.

\begin{proposition} For all $a,b,c \in \N$
$$d_{\chi}(T(ab,c), T(a,bc))=|\chi(T(ab,c))-\chi(T(a,bc))|=(b-1)|c-a|.$$
\end{proposition}

As a special case, for all $n \in \N$,
$$d_{\chi}(T(2n,n+1),T(n,2n+2))=1.$$
These are possibly the only pairs of torus links whose cobordism distance is one, apart from the obvious families $T(2,n)$, $T(2,n+1)$.

\begin{proof}[Proof of Proposition~1] Let us assume $c>a$. The maximal Euler characteristic of the knots $T(ab,c)$ and $T(a,bc)$ is $-abc+ab+c$ and $-abc+a+bc$, respectively. The absolute value of the difference of these numbers is precisely $(b-1)(c-a)$, which gives the desired lower bound for the distance $d_{\chi}(T(ab,c), T(a,bc))$. A cobordism of maximal Euler characteristic $-(b-1)(c-a)$ can be constructed by smoothing $(b-1)(c-a)$ properly chosen crossings along $(b-1)$ vertical lines of the standard diagram of the knot $T(ab,c)$. This is illustrated in Figure~4 for the two triples $(a,b,c)=(3,2,7)$ and $(a,b,c)=(2,3,7)$ on the left and right, respectively.
\end{proof}

\begin{figure}[ht]
\scalebox{0.8}{\raisebox{-0pt}{$\vcenter{\hbox{\epsffile{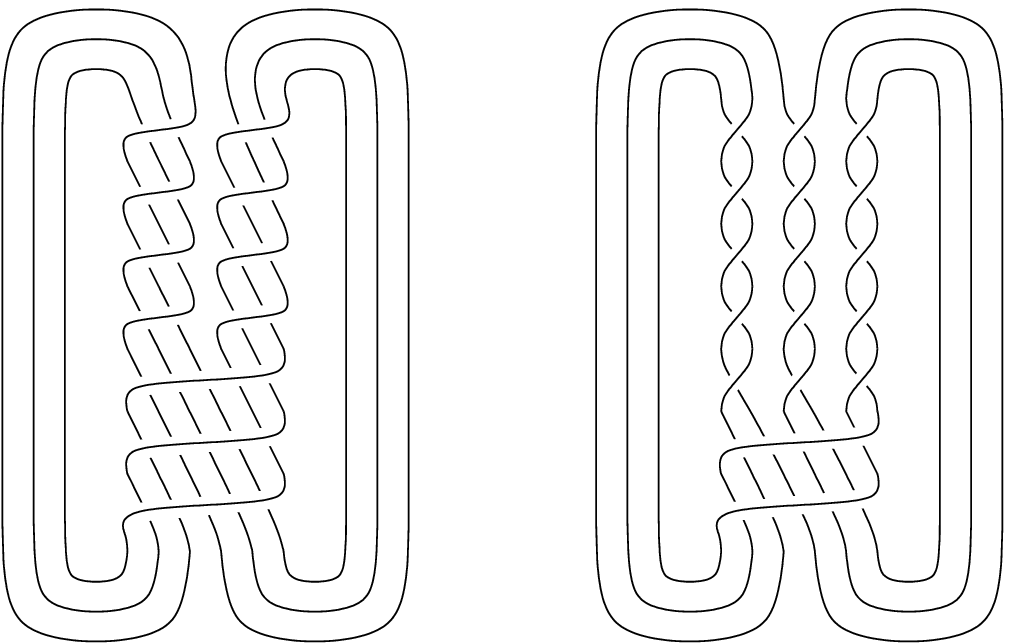}}}$}} 
\caption{}
\end{figure}

\begin{remark} It is a curious fact that the difference of the maximal Euler characteristics of the links $T(ab,c)$ and $T(a,bc)$ coincides with the difference of the sums of their parameters:
$$|\chi(T(ab,c))-\chi(T(a,bc))|=|ab+c-a-bc|.$$
\end{remark}

\section{Signatures and the cobordism distance}

The classical signature invariant can be generalised to a $1$-parameter family of invariants called Levine-Tristram signatures, via weighted Seifert matrices (\cite{Le}, \cite{Tr}). More precisely, let $\omega \in S^1 \subset \C$ be a complex number of modulus $1$ and let $M$ be a Seifert matrix of a link $L$. Then $\sigma_{\omega}(L)$ is defined as the signature of the matrix
$$(1-\omega)M+(1-\overline \omega)M^T.$$ 
The Levine-Tristram signatures may be viewed as an integer-valued step function on the unit circle with discontinuities at the roots of the Alexander polynomial of the link $L$. In the case of torus links, there is a formula for these signatures, going back to Brieskorn (\cite{Br}, see also \cite{GG}).

Let $p,q \in \N$ be natural numbers and $\omega=\exp(2 \pi i \theta)$ be a root of unity with $0<\theta<1$. Then
\begin{equation}
\sigma_{\omega}(T(p,q))=\sum_{\genfrac{}{}{0pt}{}{1 \leq x \leq q-1}{1 \leq y \leq p-1}} \epsilon_{\theta}(x,y),
\label{signature}
\end{equation}
where
$$\epsilon_{\theta}(x,y)=
\begin{cases}
+1 & \text{if $0<\theta+\frac{x}{q}+\frac{y}{p}<1 \pmod 2$,} \\
-1 & \text{if $1<\theta+\frac{x}{q}+\frac{y}{p}<2 \pmod 2$,} \\
\phantom{+} 0 & \text{if $\theta+\frac{x}{q}+\frac{y}{p} \in \Z$.} 
\end{cases}
$$

In particular, $\sigma_{\omega}(T(p,q))=0$ if $0<\theta<\frac{1}{pq}$, and the first jump of $\sigma_{\omega}$ takes place at $\theta=\frac{1}{pq}$.

Exactly as for the classical signature (which corresponds to $\omega =-1$), every $\omega$-signature yields a bound for the cobordism distance of pairs of links $L_1, L_2$ (see~\cite{Mu} for the classical signature; \cite{Tr} for $\omega$-signatures):
$$d_{\chi}(L_1,L_2) \geq |\sigma_{\omega} (L_1)-\sigma_{\omega}(L_2)|.$$

In order to be close with respect to the cobordism distance, two torus links must have similar step functions and $4$-genera. Let us recall that the the quantity $\tau$ is defined for pairs of links $K$, $L$ as 
$$\tau(K,L)=\max\{|\chi(K)-\chi(L)|,|\sigma_{\omega}(K)-\sigma_{\omega}(L)||\omega \in S^1 \}.$$
We observe that $\tau(T(a,b), T(c,d))=0$ implies $\{a,b\}=\{c,d\}$. Indeed, the first jump of the step function $\omega \mapsto \sigma_{\omega}(T(a,b))$ is at $\theta=\frac{1}{ab}$, which determines the product $ab$. The Euler characteristic $\chi(T(a,b))=-ab+a+b$ further determines the sum $a+b$, in turn the pair $\{a,b\}$. Before proving Theorem~1, let us mention two recursive formulas for the classical signature invariant of torus links, due to Gordon, Litherland and Murasugi~\cite{GLM}.
\begin{enumerate}
\item $\sigma(T(p,q+2p))=\sigma(T(p,q))-p^2+1$, if $p$ is odd, \\
\item $\sigma(T(p,q+2p))=\sigma(T(p,q))-p^2$, if $p$ is even. \\
\end{enumerate}
The following estimates are easy consequences of these.

\bigskip
\bigskip
\bigskip
\bigskip
\begin{lemma} \quad
\begin{enumerate} 
\item[(i)] $|\sigma(T(p,q))-\frac{pq}{2}-\frac{q}{2p}| \leq p$, for all $p \in 2\N+1$,\\
\item[(ii)] $|\sigma(T(p,q))-\frac{pq}{2}| \leq p$, for all $p \in 2\N$.
\end{enumerate}
\end{lemma}

\begin{proof}
By the recursive formulas,
$$\lim_{n \to \infty} \frac{1}{n} \sigma(T(p,n))=\frac{p^2-1}{2p}, \text{ if $p$ is odd},$$
$$\lim_{n \to \infty} \frac{1}{n} \sigma(T(p,n))=\frac{p}{2}, \text{ if $p$ is even},$$
in accordance with the lemma's statement. Moreover, the restriction of the signature invariant to the braid group $B_p$ with $p$ strands is a quasimorphism of defect $p$, i.e. 
$$|\sigma(\alpha \beta)-\sigma(\alpha)-\sigma(\beta)| \leq p,$$
for all braids $\alpha$, $\beta \in B_p$. Here by $\sigma(x)$ we mean the signature of the closure of the braid $x$. This follows from the fact that the signature invariant provides a lower bound for the cobordism distance of links and that there is an elementary cobordism of Euler characteristic $p$ separating the two factors of the braid $\alpha \beta$ (more details on the defect of this quasimorphism can be found in~\cite{GG}). In particular, the map $n \mapsto \sigma(T(p,n))$ is a quasimorphism on $\Z$. The required inequality now follows from a general feature of quasimorphisms on $\Z$: let $\varphi: \Z \to \Z$ be a quasimorphism of defect $D>0$ with $\lim_{n \to \infty} \frac{1}{n} \varphi(n)=a \in \R$. Then 
$$|\varphi(n)-an| \leq D$$
holds for all $n \in \N$. Indeed, suppose there exists $m \in \N$ with $|\varphi(m)-am| > D$, say $\varphi(m)=am+D+\epsilon$, for some $\epsilon >0$. Then one easily deduces
$$\varphi(km) \geq akm+D+k\epsilon,$$
for all $k \in \N$, contradicting the assumption on the limit slope of $\varphi$ (interesting information on this can be found in~\cite{AC1}).
\end{proof}

The proof of Theorem~1 is divided into three cases, depending on the difference of the braid indices. Without loss of generality, we may assume 
$$2 \leq a \leq c, \, a \leq b, \, c \leq d.$$
In particular, the braid indices of the torus links $T(a,b)$, $T(c,d)$ are $a$ and $c$, respectively. 
The first case, $c=a$, is an easy consequence of the Thom conjecture. The second case, $c \geq a+2$ is covered by Proposition~1 and estimates based on Lemma~1 and the Thom conjecture. The remaining case, $c=a+1$, needs a special consideration; it is the only instance where the cobordism distance is not determined solely by the classical signature invariant and the Euler characteristic.

\begin{proof}[Proof of Theorem 1] \quad

\noindent
i) $c=a$.
In this case
$$d_{\chi}(T(a,b), T(c,d))=|\chi(T(a,b))-\chi(T(c,d))|,$$
thus the desired inequality is true with $\gamma=1$.





\bigskip
\noindent
ii) $c \geq a+2$.
We may suppose $b \geq d$, since otherwise $d_{\chi}(T(a,b), T(c,d))=|\chi(T(a,b))-\chi(T(c,d))|$. Under the additional hypothesis
$$d \geq \max\{c^3,120c^2\},$$
we will show that the cobordism distance is bounded above by a constant multiple of
$\max \{\Delta \chi, \Delta \sigma \}$, where $\Delta \chi=|\chi(T(a,b))-\chi(T(c,d))|$, $\Delta \sigma=|\sigma(T(a,b))-\sigma(T(c,d))|$.
The assumption $d \geq \max\{c^3,120c^2\}$ leaves out finitely many exceptional torus links of braid index $c$. 

Let $k,r \in \N$ be the unique natural numbers with $d=ka+r$, $0 \leq r \leq a-1$. Our estimate for $d_\chi=d_\chi(T(a,b),T(c,d))$ is based on the following sequence of cobordisms:
$$T(c,d) \stackrel{\mathcal{C}_1}{\longrightarrow} T(c,ka) \stackrel{\mathcal{C}_2}{\longrightarrow} T(kc,a) \stackrel{\mathcal{C}_3}{\longrightarrow} T(b,a).$$
The first cobordism removes a rectangle of $r \times (c-1)$ crossings; its Euler characteristic is $-r(c-1)$. By Proposition~1, the second cobordism can be chosen to have Euler characteristic $-(k-1)(c-a)$. The third cobordism removes or adds a rectangle of crossings, $|\chi(T(kc,a))-\chi(T(b,a))|$ in number. By the triangle inequality, this number is bounded above by $|\chi(T(a,b))-\chi(T(c,d))|+|\chi(T(c,d))-\chi(T(c,ka))|+|\chi(T(c,ka))-\chi(T(kc,a))|=\Delta \chi+r(c-1)+(k-1)(c-a)$. Summing up, we obtain $d_\chi \leq \Delta \chi+2r(c-1)+2(k-1)(c-a)$, in turn
$$d_\chi \leq \Delta \chi+2ac+2\frac{d}{a}(c-a).$$
By the additional hypothesis $d \geq c^3$, we are left with 
$$d_\chi \leq \Delta \chi+4\frac{d(c-a)}{a}.$$
It remains to show that $\frac{d(c-a)}{a}$ is bounded above by a constant multiple of $\max \{\Delta \chi, \Delta \sigma \}$. 

\smallskip
\noindent
If $\Delta \sigma \leq \frac{d(c-a)}{40a}$ then Lemma~1 implies $|ab-cd| \leq \Delta \sigma+|\frac{b}{a}-\frac{d}{c}|+2a+2c \leq \frac{d(c-a)}{40a}+\frac{b}{a}+4c$, thus
$$b \geq \frac{cd}{a}-\frac{b}{a^2}-\frac{d(c-a)}{40a^2}-\frac{4c}{a},$$
$$b \geq \frac{acd}{a^2+1}-\frac{d(c-a)}{40a^2}-\frac{4c}{a}.$$
Here we could have multiplied the last two terms by $\frac{a^2}{a^2+1}$, but this does not matter for us.
The fact that $b$ is roughly the same as $\frac{cd}{a}$ allows us to deduce a reasonable estimate for $\Delta \chi$:
\begin{equation*}
\begin{aligned}
\Delta \chi & \geq a+b-c-d-|ab-cd| \\
& \geq a+b-c-d-\frac{d(c-a)}{40a}-\frac{b}{a}-2a-2c \\
& \geq \frac{(a-1)}{a}b-d-\frac{d(c-a)}{40a}-4c 
\end{aligned}
\end{equation*}
The essential term of the last line is $\frac{(a-1)}{a}b-d$. By the above inequality for $b$, it is bounded below by $\frac{(a-1)cd-(a^2+1)d}{a^2+1}-\frac{d(c-a)}{40a^2}-\frac{4c}{a}$. Our goal is to bound the term $\frac{(a-1)cd-(a^2+1)d}{a^2+1}$ by a constant multiple of $\frac{d(c-a)}{a}$. This fails precisely if $(a-1)c-(a^2+1)<0$, i.e. if $c=a+1$, and in three more special cases: $(a,c)=(2,4)$, $(3,5)$ and $(2,5)$. In all other cases we have
$$\frac{(a-1)cd-(a^2+1)d}{a^2+1} \geq \frac{d(c-a)}{10a}$$
(the worst remaining case is $(a,c)=(2,6)$, in which equality holds). Altogether, we obtain:
$$\Delta \chi \geq \frac{d(c-a)}{10a}-\frac{d(c-a)}{40a^2}-\frac{4c}{a}-\frac{d(c-a)}{40a}-4c.$$
Under the additional hypothesis $d \geq 120c^2$, the term $\frac{4c}{a}+4c \leq 6c$ is bounded above by $\frac{d(c-a)}{40a}$; we are left with an estimate of the desired type:
$$\Delta \chi \geq \frac{d(c-a)}{40a}.$$ 
The three exceptional cases $(a,c)=(2,4)$, $(3,5)$ and $(2,5)$ can be treated by a more careful analysis of the classical signature invariant. Details are not interesting enough to be presented here.

\bigskip
\noindent
iii) $c=a+1$.
The same construction as in the previous case shows
$$d_\chi \leq \Delta \chi+4\frac{d(c-a)}{a}=\Delta \chi+\frac{4d}{a}.$$
However, the term $\frac{4d}{a}$ is not bounded above by a constant multiple of $\max \{\Delta \chi, \Delta \sigma \}$. The reason for that lies in the asymptotic behaviour of the ratio $\frac{\sigma}{\chi}$ for torus links of type $(p,n)$, where $n$ tends to infinity. These ratios can be determined via the recursive formulas for the classical signature invariant:
$$\lim_{n \to \infty} \frac{\sigma(T(p,n))}{\chi(T(p,n))}=\frac{p^2-1}{2p(p-1)}=\frac{p+1}{2p}, \text{ if $p$ is odd},$$
$$\lim_{n \to \infty} \frac{\sigma(T(p,n))}{\chi(T(p,n))}=\frac{p^2}{2p(p-1)}=\frac{p}{2(p-1)}, \text{ if $p$ is even}.$$
The limits coincide for $a=p$, $c=p+1$, provided $p$ is odd. We will therefore need the whole spectrum of $\omega$-signature invariants.

Using the sum formula~(\ref{signature}), we may analyse the coarse profile of the $\omega$-signature function $\sigma_{\omega}(T(p,n))$, for large numbers $n$ (compare~\cite{GG}). If $\theta$ is small, more precisely if $\theta < \frac{1}{p}$, the piecewise constant function $\theta \mapsto \sigma_{\exp(2 \pi i \theta)}(T(p,n))$ jumps by $-2$ at all multiples of $\frac{1}{pn}$ except at multiples of $\frac{p}{pn}=\frac{1}{n}$, where there are no jumps. This implies
$$\lim_{n \to \infty} \frac{1}{n} \, \sigma_{\exp(2 \pi i \frac{1}{p})}(T(p,n))=-\frac{2(p-1)}{p}.$$
For $\frac{1}{p} < \theta < \frac{2}{p}$, the value $\theta+\frac{x}{q}+\frac{y}{p}$ can exceed $2$. More precisely, the function $\theta \mapsto \sigma_{\exp(2 \pi i \theta)}(T(p,n))$ jumps by $2$ at intervals of length $\frac{1}{n}$, starting at $\theta=\frac{1}{p}+\frac{1}{n}$. This implies
$$\lim_{n \to \infty} \frac{1}{n} \, \sigma_{\exp(2 \pi i \frac{2}{p})}(T(p,n))=-2(\frac{p-1}{p}+\frac{p-3}{p}).$$
Continuing in this way, we see that the asymptotic profile of the $\omega$-signature function $\sigma_{\exp(2 \pi i \theta)}(T(p,n))$ between $\theta=0$ and $\theta=\frac{1}{2}$ is piecewise affine with consecutive slopes $-\frac{2(p-1)}{p}, -\frac{2(p-3)}{p}, -\frac{2(p-5)}{p}, \ldots$.
At this point we observe a fundamental difference between the profiles for even and odd numbers $p$: for even $p$, the slope of the final segment is $-\frac{2}{p}$, whereas for odd $p$, it is zero. 

\smallskip
\noindent
Assuming that $a$ is odd (for notational convenience), we obtain
$$|\sigma_{\exp(2\pi i (\frac{1}{2}-\frac{1}{a+1}))}(T(a,b))-\sigma_{\exp(\pi i)}(T(a,b))| \leq a,$$
$$|\sigma_{\exp(2\pi i (\frac{1}{2}-\frac{1}{a+1}))}(T(a+1,d))-\sigma_{\exp(\pi i)}(T(a+1,d))| \geq \frac{2d}{a+1}-(a+1).$$
As a consequence, there exists an $\omega$ ($\exp(2\pi i (\frac{1}{2}-\frac{1}{a+1}))$ or $\exp(\pi i)$, to be precise) with
$$\Delta \sigma_\omega=\sigma_\omega(T(a,b))-\sigma_\omega(T(a+1,d)) \geq \frac{d}{a+1}-(2a+1).$$
Under the additional hypothesis $d \geq 8c^2$, the term $2a+1$ is bounded above by $\frac{d}{4a}$, thus
$$\Delta \sigma_\omega \geq \frac{d}{4a},$$
which is precisely what we need.
\end{proof}

\section{Searching for the optimal linear bound}

The purpose of this section is to show that the linear bound of Theorem~2,
$$f(a,b,c,d) \leq 2(a+b+c+d),$$
is almost optimal. Evidently we are only interested in bounds that respect all symmetries of the parameters, i.e. bounds that are invariant under exchanging $a$ and $b$, $c$ and $d$, as well as the pairs $\{a,b\}$ and $\{c,d\}$. The only `linear' bounds with these properties are of the form $\alpha (a+b+c+d)$ or $\beta |a+b-c-d|$, for constants $\alpha, \beta \in \R$. Let us first show $\alpha \geq \frac{1}{2}$, by looking at a particular family of examples.  

Fix a natural number $n$ and set $a=2$, $b=n^2+1$, $c=d=n+1$. Then the quantity $|\chi(T(a,b))-\chi(T(c,d))|$ is zero and the expression $\alpha (a+b+c+d)$ is essentially $\alpha n^2$, for large $n$. Up to a linear error in $n$, the values of the classical signature invariant on the links $T(2, n^2+1)$ and $T(n+1,n+1)$ are $-n^2$ resp. $-\frac{n^2}{2}$. Thus the cobordism distance of the links $T(2, n^2+1)$ and $T(n+1,n+1)$ is at least $\frac{n^2}{2}$, implying $\alpha \geq \frac{1}{2}$.

Finding a family of links that excludes all bounds of the form $\beta |a+b-c-d|$ is somewhat more subtle. \footnote{A careful inspection of the proof of Theorem~1 reveals that such a bound exists, provided we exclude finitely many exceptional torus links per braid index.}
Fix a natural number $n$ and set $a=6n$, $b=12n+1$, $c=6n+1$, $d=12n-1$. The quantity $|\chi(T(a,b))-\chi(T(c,d))|$ is again zero. Furthermore $a+b-c-d=1$, so the expression $\beta |a+b-c-d|$ coincides with $\beta$, for all $n$. The signature of the involved torus links can be computed by using recursive formulas~\cite{GLM}. 
The outcome is
$$\sigma(T(6n, 12n+1))=-36n^2,$$
$$\sigma(T(6n+1, 12n-1))=-36n^2-4n.$$
The difference exceeds every constant $\beta$.

\section{Stable $4$-genus}

The smooth $4$-genus $g_4(K)$ of a knot $K \subset S^3$ is the minimal genus among all smooth oriented surfaces without closed components in the $4$-ball with boundary $K$. Contrary to the classical genus, the $4$-genus is not additive, but sub-additive, under the connected sum of knots:
$$g_4(K \# L) \leq g_4(K)+g_4(L),$$
for all knots $K,L \subset S^3$. In~\cite{Li}, Livingston introduced the stable $4$-genus of a knot $K$ as
$$g_{st}(K)=\lim_{n \to \infty} \frac{g_4(nK)}{n} \in \R.$$
He further showed that the natural extension of the stable $4$-genus to the rationalised concordance group $\mathcal{C}_{\Q}$ is a semi-norm. It is instructive to study this semi-norm on small sub-groups of $\mathcal{C}_{\Q}$, for example sub-groups spanned by pairs of knots. This was done by Livingston for certain pairs of torus knots. Here we will explain why the restriction of the stable $4$-genus to the span of pairs of torus knots has flat unit balls.

Let $T(a,b)$, $T(c,d)$ be two torus knots with positive parameters $a,b,c,d$. We interpret a point $(x,y) \in \Q^2$ as the element
$$xT(a,b)+yT(c,d) \in \mathcal{C}_{\Q}.$$
By the Thom conjecture, the stable $4$-genus of a positive torus knot of type $(p,q)$ is $g_{st}(T(p,q))=\frac{1}{2}(p-1)(q-1)$. Thus the points $(1,0)$ and $(0,1)$ have norms $\frac{1}{2}(a-1)(b-1)$ and $\frac{1}{2}(c-1)(d-1)$, respectively. Now let us compute the norm of the points $(-1,1)$ and $(1,-1)$. By Theorem~2, there exists a smooth cobordism of genus at most
$$\frac{1}{2}|(a-1)(b-1)-(c-1)(d-1)|+a+b+c+d$$
between the two knots $T(a,b)$ and $T(c,d)$ (the genus of a cobordism between two knots equals half the absolute value of its Euler characteristic). This is then an upper bound for the norm of the points $(-1,1)$ and $(1,-1)$. Assuming that the difference of the genera of the knots $T(a,b)$ and $T(c,d)$ is small compared to the quantity $a+b+c+d$, this implies that the unit ball with respect to the stable $4$-genus is flat, with a long diameter along the line of slope $-1$.

\begin{example}
Let $(a,b,c,d)=(5,8,4,11)$. We have
$$g_{st}(T(5,8))=\frac{1}{2}(5-1)(8-1)=14,$$
$$g_{st}(T(4,11))=\frac{1}{2}(4-1)(11-1)=15.$$
By Proposition~1, there is a cobordism of Euler characteristic $-1$ between the knot $T(5,8)$ and the link $T(4,10)$. Further there is an obvious cobordism of Euler characteristic $-3$ between the link $T(4,10)$ and the knot $T(4,11)$. Altogether, there is a cobordism of Euler characteristic $-4$, i.e. of genus $2$, between the knots $T(5,8)$ and $T(4,11)$. The expected shape of the corresponding norm ball is sketched in Figure~5. It could actually be even flatter, if the cobordism distance between the knots $T(5,8)$ and $T(4,11)$ were two.
\end{example}

\begin{figure}[ht]
\scalebox{1.0}{\raisebox{-0pt}{$\vcenter{\hbox{\epsffile{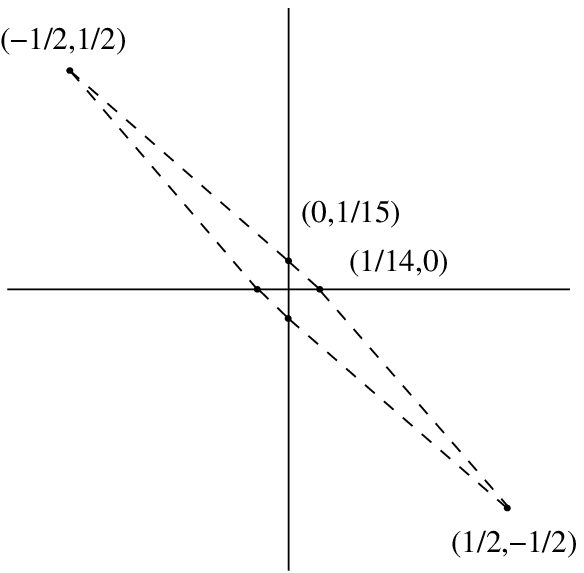}}}$}} 
\caption{}
\end{figure}

\bigskip
\noindent
Universit\"at Bern, Sidlerstrasse 5, CH-3012 Bern, Switzerland

\bigskip
\noindent
\texttt{sebastian.baader@math.unibe.ch}


\begin{thebibliography}{99} 

\bibitem{AC1}
     N.~A'Campo: \emph{A natural construction for the real numbers}, arXiv: math/ 0301015, 2003. 

\bibitem{Br}
     E.~Brieskorn: \emph{Beispiele zur Differentialtopologie von Singularit\"aten}, Invent. Math.~\textbf{2} (1966), 1-14.

\bibitem{GG}
     J.M.~Gambaudo, E.~Ghys: \emph{Braids and signatures}, Bull. Soc. Math. France~\textbf{133} (2005),  no.~4, 541-579. 

\bibitem{GLM}
     C.McA.~Gordon, R.A.~Litherland, K.~Murasugi: \emph{Signatures of covering links}, Canad. J.~Math \textbf{33} (1981), no.~2, 381-394.

\bibitem{HU}
     M.~Hirasawa, Y.~Uchida: \emph{The Gordian complex of knots}, J. Knot Theory Ramifications \textbf{11} (2002), no.~3, 363-368.

\bibitem{Ka}
     T.~Kawamura: \emph{On unknotting numbers and four-dimensional clasp numbers of links}, Proc. Amer. Math. Soc. \textbf{130} (2002), no.~1, 243-252.


\bibitem{KM}
     P.~B.~Kronheimer, T.~S.~Mrowka: \emph{The genus of embedded surfaces in the projective plane},  Math. Res. Lett. \textbf{1} (1994), no.~6, 797-808. 

\bibitem{Le}
     J.~Levine: \emph{Knot cobordism groups in codimension two}, Comment. Math. Helv. \textbf{44} (1969), 229-244.

\bibitem{Li}
     C.~Livingston: \emph{The stable 4-genus of knots}, arXiv:0904.3054, 2009.

\bibitem{Mu}
     K.~Murasugi: \emph{On a certain numerical invariant of link types}, Trans. Amer. Math. Soc.~\textbf{117} (1965), 387-422.

\bibitem{Ra}
     J.~Rasmussen: \emph{Khovanov homology and the slice genus}, Invent.~\textbf{182} (2010), 419-447, arXiv: math/0402131.

\bibitem{Tr}
     A.~C.~Tristram: \emph{Some cobordism invariants for links}, Proc. Cambridge Philos. Soc. \textbf{66} (1969), 251-264.  

\end{thebibliography}
\end{document}